\documentclass{amsart}
\usepackage[title]{appendix}
\usepackage[a4paper,top=3cm,bottom=2cm,left=3cm,right=3cm,marginparwidth=1.75cm]{geometry}

\usepackage{url} 
\usepackage{hyperref}
\usepackage[noabbrev,capitalise]{cleveref}
\usepackage[utf8]{inputenc}
\usepackage[T1]{fontenc}
\usepackage{amsmath,amssymb,amsthm,amsfonts,mathrsfs,mathtools,relsize,comment}
\usepackage{xcolor}
\usepackage{thm-restate}

\title{The Burning Number Conjecture Holds Asymptotically}

\newtheorem{thm}{Theorem}[section]

\newtheorem{conjecture}[thm]{Conjecture}
\newtheorem{lem}[thm]{Lemma}

\newtheorem{question}[thm]{Question}
\newtheorem*{theorem*}{Theorem}

\theoremstyle{remark}

\newtheoremstyle{named}{}{}{\itshape}{}{\bfseries}{.}{.5em}{\thmnote{#3 }#1}
\theoremstyle{named}

\makeatletter
\@namedef{subjclassname@2020}{%
\textup{2020} Mathematics Subject Classification}
\makeatother


\usepackage[shortlabels]{enumitem}

\newcommand{\mc}[1]{\mathcal{#1}}

\newcommand{\bb}[1]{\mathbb{#1}}
\newcommand{\bl}[1]{\boldsymbol{#1}}

\newcommand{\eps}{\varepsilon}

\newcommand{\s}[1]{\left(#1\right)}

\newcommand{\m}{\textup{\textsf{m}}}

\newcommand{\R}{\mathbb{R}}

\newcommand{\N}{\mathbb{N}}

\newcommand{\E}{\mathbb{E}}
\renewcommand{\P}{\mathbb{P}}

\DeclareMathOperator{\dpt}{dp}
\DeclareMathOperator{\diam}{diam}
\DeclareMathOperator{\anc}{anc}
\DeclareMathOperator{\comp}{\bar {comp}}

\let\L\relax
\DeclareMathOperator{\L}{L}
\DeclareMathOperator{\Br}{Br}

\begin{document}

\title{The Burning Number Conjecture Holds Asymptotically}

\author{Sergey Norin}
\address{Department of Mathematics and Statistics, McGill University, Montr\'eal, Canada}
\email{snorin@math.mcgill.ca}
\urladdr{www.math.mcgill.ca/snorin/}

\author{J\'er\'emie Turcotte}
\address{Department of Mathematics and Statistics, McGill University, Montr\'eal, Canada}
\email{mail@jeremieturcotte.com}
\urladdr{www.jeremieturcotte.com}

\thanks{The authors are supported by the Natural Sciences and Engineering Research Council of Canada (NSERC). Les auteurs sont support\'es par le Conseil de recherches en sciences naturelles et en g\'enie du Canada (CRSNG)}

\begin{abstract}
	The burning number $b(G)$ of a graph $G$ is the smallest number of turns required to burn all vertices of a graph if at every turn a new fire is started and existing fires spread to all adjacent vertices. The Burning Number Conjecture of Bonato et al. (2016) postulates that $b(G)\leq \left\lceil\sqrt{n}\right\rceil$ for all graphs $G$ on $n$ vertices. We prove that this conjecture holds asymptotically, that is $b(G)\leq (1+o(1))\sqrt n$.
\end{abstract}

\maketitle

\section{Introduction}

Consider the following process (or one-player game) on a finite, usually connected, graph $G$. At the start, all vertices are said to be unburned. Then, at every step (or turn) of the process, we may choose to burn (start a new fire) at some vertex. Furthermore, at every step, all vertices adjacent to burning vertices will start burning as well (the fire spreads to neighbouring vertices at every turn). How many steps does it take before all vertices of the graph are burned?

The \emph{burning number} of $G$, denoted by $b(G)$, is the smallest number of turns required to burn the entire graph $G$. Formally, $b(G)$ is the smallest integer $k$ such that we can cover $G$ with balls of radii $0,\dots,k-1$.

This process first appeared in print in a paper of Alon \cite{alon_transmitting_1992} and was motivated by a question of Brandenburg and Scott at Intel, and formulated as a transmission problem involving a set of processors. Alon \cite{alon_transmitting_1992} proved that the burning number of a $d$-dimension hypercube is exactly $\left\lceil\frac{d}{2}\right\rceil+1$.

This process was later independently defined by Bonato et al. \cite{bonato_how_2016} (also in the related \cite{bonato_burning_2014,roshanbin_burning_2016}), and named \emph{graph burning}. In this case, the inspiration is the spread of information (news, memes, opinions, etc.) in social networks. Since its reintroduction only a few years ago, a large body of work on graph burning has emerged, concentrating mainly on bounding the burning number, either in general or on specific classes of graphs, and on the complexity of graph burning. See \cite{bonato_survey_2021} and references therein for a fairly recent survey of results on graph burning.

Bonato et al. \cite{bonato_how_2016} show that $b(P_n)=\left\lceil\sqrt{n}\right\rceil$ and conjecture that paths are in fact the connected graphs with the largest burning number.

\begin{conjecture}[Burning Number Conjecture]\label{conj:burningnumberconjecture}\cite{bonato_how_2016}
	If $G$ is a connected graph on $n$ vertices, then $b(G)\leq \left\lceil\sqrt{n}\right\rceil$.
\end{conjecture}

Since the introduction of graph burning, this conjecture has been the central open problem in the field. The conjecture is known to hold on multiple classes of graphs, such as spiders \cite{das_burning_2018,bonato_bounds_2019}, caterpillars \cite{liu_burning_2020}, some $p$-caterpillars \cite{hiller_burning_2021}, sufficiently large graphs with minimum degree at least 4 \cite{bastide_improved_2022}, and others (for instance \cite{omar_burning_2021}).
 
The best known bound on the burning number of general graphs has been improved multiple times. Suppose $G$ is a connected graph on $n$ vertices.
\begin{itemize}
	\item Bonato et al. \cite{bonato_how_2016} show that $b(G)\leq 2\left\lceil\sqrt{n}\right\rceil-1$.
	\item Bessy et al. \cite{bessy_bounds_2018} show that $b(G)\leq\sqrt{\frac{32n}{19(1-\varepsilon)}}+\sqrt{\frac{27}{19\varepsilon}}$ for every $0<\varepsilon<1$ and $b(G)\leq \sqrt{\frac{12n}{7}}+3$.
	\item Land and Lu \cite{land_upper_2016} show that $b(G)\leq \left\lceil\frac{\sqrt{24n+33}-3}{4}\right\rceil$.
	\item Bastide et al. \cite{bastide_improved_2022} show that $b(G)\leq \left\lceil\sqrt{\frac{4n}{3}}\right\rceil+1$.
\end{itemize}

Our main result is the following, which shows that \cref{conj:burningnumberconjecture}  holds asymptotically.
\begin{restatable}{thm}{Main}\label{thm:main}
     If $G$ is a connected graph on $n$ vertices, then $$b(G)\leq (1+o(1))\sqrt{n}.$$  
\end{restatable}

Bastide et al. \cite[Theorems 2-3]{bastide_improved_2022} prove that the Burning Number Conjecture holds for sufficiently large graphs with minimum degree at least four, and almost holds (up to adding some small constant) for graphs with minimum degree three. We note that with their argument, one can deduce from \cref{thm:main} that the Burning Number Conjecture also holds for sufficiently large graphs with minimum degree three. In fact, for these graphs we can get the bound $b(G)\leq \left(\sqrt{\frac{3}{4}}+o(1)\right)\sqrt{n}$ for graphs with minimum degree three, and $b(G)\leq \left(\sqrt{\frac{3}{5}}+o(1)\right)\sqrt{n}$ for graphs with minimum degree at least four. 

The remainder of the paper is occupied by the proof of \cref{thm:main}, which we now outline. 

As noted in \cite{bonato_how_2016}, if $T$ is a spanning tree of $G$, then $b(G)\leq b(T)$; a strategy to burn $T$ still works when edges are added. Hence, it suffices to prove \cref{thm:main} for trees and we focus on trees from now on. 

The results on the burning number mentioned above largely use traditional graph theoretical techniques, i.e. induction.

While we borrow many ideas from the previous work, we deviate from this pattern by shifting to a continuous and probabilistic setting. First, as the burning number is defined in terms of existence of a cover of the graphs by balls with certain radii, we embrace the metric nature of the problem and work with metric, rather than discrete, trees, i.e. metric spaces obtained by replacing the edges of a tree by real intervals.  This has several advantages, in particular allowing us to ignore the scaling issues. \cref{s:metricTrees} is devoted to introducing the setting.
In \cref{s:covers} we prove that existence of covers of any given metric tree by balls, which are  ``frugal'', in a sense that the sum of the radii of balls is about as small as can be expected, and flexible, allowing us to choose from many different possible radii.  In \cref{s:random} we bootstrap the results of \cref{s:covers} to prove \cref{t:main2} - a fractional version of \cref{thm:main}. Informally, \cref{t:main2} says that for every metric tree $T$ and every $r$ there exists a distribution on covers of $T$ by balls of radii at most $r$ which is still frugal in expectation and such that the distribution is almost uniform - uses all radii with roughly the same probability. 
In  \cref{s:proof} we use \cref{t:main2} to prove a version of \cref{thm:main} for metric trees and then transfer the result back to the discrete setting. The key observation here is that if we divide the tree $T$ into many pieces and independently choose a cover of each piece with the properties guaranteed by \cref{t:main2} then the radii in the resulting cover will be almost uniformly distributed, i.e. close to the desired set of radii $\{1,2,\ldots \left\lceil\sqrt{n}\right\rceil\}$.  

We conclude in \cref{s:remark} with a few remarks, in particular briefly discussing potential generalization of the metric version of the Burning Number Conjecture that allows less restrictive sets of radii.
	
\section{Metric trees}\label{s:metricTrees}

As mentioned in the introduction, we will primarily be working with metric trees and we start this section by formally defining them and introducing the necessary notation.

A \emph{metric tree} $T$ is obtained from a non-null (discrete) tree $T_0$ by replacing each edge  with a close interval of positive length, 
so that the end points of the interval are identified with the end vertices of the edge. 
The \emph{length} of $T$, denoted  by $|T|$, is the sum of the length of the intervals comprising it.
Note that for any two points $u,v$ in a metric tree $T$ there exists a unique path (simple curve) in $T$ 
with ends $u$ and $v$, which we denote by $T[u,v]$, and we denote its length by $d_T(u,v)$ or simply by $d(u,v)$. It is easy to see that $(T,d_T)$ is a metric space. In fact, $T$ is also compact, since it is obtained from identifying the ends of a finite number of closed (compact) intervals\footnote{Given an open cover of $T$, one can find a finite subcover by taking the union of the finite subcovers obtained from the compactness of each interval.}.

A \emph{leaf} of $T$ is a leaf of $T_0$, and a \emph{branch point} of $T$ is a vertex of $T_0$ with degree three or greater. Leaves and branch points of a metric tree $T$ can be also defined intrinsically, i.e.
a point $v \in T$ is a leaf if and only if $T \setminus v$ is connected, and a branch point of $T$ is a point of $T$ such that $T \setminus v$ has at least three connected components. Let $L(T)$ and $\Br(T)$ denote the sets of leafs and branch points of $T$, respectively. The components of $T\setminus (L(T) \cup \Br(T))$ are homeomorphic to open intervals. The \emph{segments} of $T$ are the closures of these intervals. Thus every point in $T \setminus{\Br(T)}$  belongs to a unique segment, and the ends of every segment lie in  $L(T) \cup \Br(T)$.

The \emph{diameter} of $T$, denoted by $\diam(T)$, is the maximum distance between two points in $T$, that is$$\diam(T)=\max_{u,v\in T}d_T(u,v).$$ This is well defined by the compactness of $T$ (and given the continuity of $d_T$).

A metric tree $T$ is  \emph{trivial} if $|T| = 0$, in which case $T$ consists of a single point, and $T$ is \emph{non-trivial}, otherwise. A metric tree $T'$ is a \emph{subtree} of a metric tree $T$ if $T' \subseteq T$. For instance, for any $u,v\in T$, $T[u,v]$ is a subtree of $T$ with 2 leaves, unless $u=v$ in which case $T[u,v]$ is trivial. A collection  $\mc{T}=\{T_1,\ldots,T_k\}$ of subtrees of a metric tree $T$ is \emph{a decomposition of $T$} if $T = \cup_{i \in [k]} T_i$ and $T_1,T_2,\ldots, T_k$ are pairwise internally disjoint, i.e. $|T| = \sum_{i=1}^k |T_i|$.

For a subtree $T'$ of a metric tree $T$ we denote by $\bar{T'}$ the closure of $T \setminus T'$. We denote by $\comp(T')$ the set of components of $\bar{T'}$. Note that every component of $\bar{T'}$ is a subtree of $T$, that $\{T'\} \cup  \comp(T')$ is a decomposition of $T$, and that every component of $\bar{T'}$ shares  exactly one point with $T'$.

A subtree $T'$ of  $T$ is \emph{a branch of $T$} if $\bar{T'}$ is connected and non-empty, i.e.   $\bar{T'}$
is a subtree of $T$. It follows from the observations above that for every proper subtree $T'$ of a metric tree $T$
every component of $\bar{T'}$ is a branch of $T$.  The \emph{anchor} $\anc(T')$ of a branch  $T'$ of $T$ is the unique point of $T' \cap \bar{T}'$. 
The \emph{depth} $\dpt(T')$ of a branch $T'$ of $T$ is defined as $$\dpt(T') = \max_{v \in T'}d_T(\anc(T'),v).$$
Again, this is well defined by compactness.

For some minimality arguments, we will need a stronger compactness property. For a metric tree $T$, let $\mathcal S(T)$ be the set of subtrees of $T$. It is easy to see that $\mathcal S(T)$ is the set of closed (hence compact) and connected (in this case as in $\R$, this is equivalent to being path-connected) subsets of $T$. Let $d_{\mathcal S(T)}$ be the Hausdorff distance \cite{hausdorff_set_1957} on $\mathcal S(T)$, formally if $T_1,T_2$ are subtrees of $T$, $$d_{\mathcal S(T)}(T_1,T_2)=\max\left(\max_{v_1\in T_1}d_T(v_1,T_2),\max_{v_2\in T_2}d_T(v_2,T_1)\right),$$
where, for $v\in T$ and a subtree $T'$, $d_T(v,T')=\min_{v'\in T'} d_T(v,v')$. These are well defined (we can write max, min instead of the usual sup, inf) by compactness of $T$. In fact, $(\mathcal S(T),d_{\mathcal S(T)})$ is a compact metric space; the space of non-empty compact subsets of a compact set with this metric is compact, and $\mathcal S(T)$ is a closed subspace of this space given that the limit of connected sets is connected \cite{hausdorff_set_1957} (these properties are fairly straightforward to prove).

Let us now prove some basic properties of metric trees.
 
\begin{lem}\label{l:diamDecomp}
	 If $z  \geq0$ and  $T'$ is a subtree of a metric tree $T$ such that  $\dpt(J) \leq z$ for every $J\in \comp(T')$, then $$\diam(T) \leq \diam(T')+2z.$$
\end{lem}	
\begin{proof}
	Let $u_1,u_2 \in T$ be such that $d_T(u_1,u_2) = \diam(T)$.  For $i=1,2$, if $u_i \not \in T'$, let  $T_i$ be the component of $\bar{T'}$ such that $u_i \in T_i$, and otherwise, let $T_i=\{u_i\}$ be a trivial branch of $T$; in particular, we always have that $\anc(u_i)\in T'$.
	Then \begin{align*}
	\diam(T) = d_T(u_1,u_2)
	&\leq d_T(u_1,\anc(T_1))+ d_T(\anc(T_1),\anc(T_2))+ d_T(\anc(T_2),u_2) \\
	&\leq \dpt(T_1)  + \diam(T') + \dpt(T_2)\\
	&\leq \diam(T')+2z,
	\end{align*}
	as desired.
\end{proof}

\begin{lem}\label{l:depth1}
	If $T',T''$ are branches of a metric tree $T$ such that $T'' \subseteq T'$, then $$\dpt(T') \geq \dpt(T'') + d_T(\anc(T'),\anc(T'')).$$
\end{lem}	
\begin{proof}	
	Since $T''\subseteq T'$, we have that $\bar T'\subseteq \bar T''$. Hence, we know that $\anc(T')\in \bar T'\subseteq \bar T''$. If $\anc(T')\in T''$, then $\anc(T')\in T''\cap \bar T''$ and so $\anc(T')=\anc(T'')$.
	
	The statement is trivial in this case. Hence we can suppose that $\anc(T')\notin T''$. Let $u\in T''$ such that $\dpt(T'')=d_T(\anc(T''),u)$. Since $\anc(T')\notin T''$, any path between $u$ and $\anc(T')$ must go through the unique point of $T''\cap \bar T''$, which is $\anc(T'')$. Hence, $d_T(\anc(T'),u)=d_T(\anc(T'),\anc(T''))+d_T(\anc(T''),u)$.

	We may thus conclude that
	\begin{align*}
		\dpt(T')&=\max_{v \in T'}d_T(\anc(T'),v)\\
		&\geq d_T(\anc(T'),u)\\
		&=d_T(\anc(T''),u)+d_T(\anc(T'),\anc(T''))\\
		&=\dpt(T'')+d_T(\anc(T'),\anc(T'')),
	\end{align*}
	as desired.
\end{proof}

Given a metric tree $T$, the following lemma will allow us in \cref{s:covers} to focus the most technical part of our analysis on a fairly simple subtree  $T'$ of $T$ which  has most 3 leaves.

\begin{lem}\label{l:z} If $T$ is a metric tree, then there exists $z \geq 0$ and a subtree $T'$ of $T$ such that
	\begin{itemize}
		\item[(i)] $T'$ has at most three leaves,
		\item[(ii)] $\dpt(J) \leq z$ for every   $J \in \comp{T}'$,  
		\item[(iii)] for every  $v \in \L(T')$  there exists $J \in \comp{T}'$ such that $\anc(J)=v$ and $\dpt(J) = z,$ 
		\item[(iv)] $|T| \geq |T'|+4z$. 
	\end{itemize} 
\end{lem}

\begin{proof} 
	Let $(T',z)$ satisfying properties (ii)-(iv) be chosen with $z$ maximum. This is possible given that the subspace of $\mathcal S(T)\times \left[0,\frac{|T|}{4}\right]$ of possible solutions is closed hence compact (for instance, using the continuity of $|\cdot|$ and of $\max_{J\in \comp(\cdot)} \dpt(J)$). Note that $(T,0)$ satisfies these properties, and so this subspace is nonempty.
	
	If $T'$ has at most three leaves then the lemma holds, and so we assume for a contradiction that $T'$ has at least four leaves.
	
	Let $\delta > 0$ be such that the length  of every segment of $T'$ is greater than $\delta$. For every  $v\in \L(T')$, let
	$p_v \in T'$ be chosen so that $p_v$ belongs to the unique segment of $T'$ with end $v$ and $d_{T'}(v,p_v)=\delta$. Let $T''$ be the subtree of $T'$ such that $p_v  \in \L(T'')$ for every leaf $v\in L(T')$. In other words, $T''$ is obtained from $T'$ by deleting for each leaf the half-open interval of length $\delta$ containing it.   
	
	We claim that $z'=z + \delta$ and $T''$ satisfy properties (ii)-(iv). Note that this claim contradicts the choice of $z$ and $T'$, and thus implies the lemma. 
	
	For every   $J \in \comp{T''}$  either \begin{itemize}
		\item $J \in \comp{T'}$ , or
		\item $\anc(J)= p_v$ for some $v \in \L(T')$ and  $J$ is the union of $T[v, p_v]$ and all the components of $\bar{T'}$ with anchors in  $T[v, p_v]$.
	\end{itemize} 
	In the first case we have $\dpt(J) \leq z<z'$ by the choice of $(T',z)$ . 
	
	Now suppose that $J$ is a  component  of  $\bar{T}''$ satisfying the conditions of the second case.
	Then $\dpt(J) \geq z+\delta=z'$ by \cref{l:depth1}. Moreover, for every $u \in J$ either $u \in T'$, in which case $d_T(p_v,u) \leq \delta \leq z'$, or there exists $J' \in \comp{T'}$  such that $u \in J'$ and $ \anc(J') \in T[v, p_v]$, in  which case $d_T(p_v,u) \leq d_T(p_v, \anc(J')) + \dpt(J') \leq \delta+z=z'$. It follows that $\dpt(J) \leq z'$, and so $\dpt(J) = z'$. 
	Therefore properties (ii) and (iii) hold; for the latter, note that every $p_v\in \L(T'')$, there is a component of $\comp{T''}$ of the second form (the component containing $p_v$ itself).
	
	Finally, $|T''|= |T'|- \delta|\L(T')|$. As $|\L(T')| \geq 4$ by our assumption, we have $|T| \geq |T'|+4z \geq |T''| + 4(z+\delta)=|T''|+z'$. Thus (iv) also holds.
\end{proof}

A key part of our argument involves dividing the tree $T$ into pieces so that  all but one of them have length at least a certain threshold $l$, and none of them are much bigger. The precise definition that works for our purposes is the following.

For a metric tree $T$ and $l \geq 0$, we say that $T$ is \emph{$l$-minimal} if $|T| \geq l$, and there exists a decomposition $\{T',T''\}$ of $T$, such that $|T'| \leq l$ and $|T''|\leq  l$. Note that this implies that $|T|\leq 2l$.

\begin{lem}\label{l:partition1}
If $l>0$ and $T$ is a metric tree such that $|T| \geq l$, then there exists a decomposition $\{T_0,T_1\}$ of $T$ such that $T_1$ is $l$-minimal.
\end{lem}	

\begin{proof} 
	Let a decomposition $\{T_0,T_1\}$ of $T$ be chosen so that $|T_1| \geq l$ and subject to this $T_1$ is minimal. This is possible given that $\mathcal S(T)$ is compact (consider the subspace of $\mathcal S(T)$ defined by valid $T_1$, it is not too hard to see that it is closed, in particular using the continuity of $|\cdot|$). We wish to show that $T_1$ is $l$-minimal. Let $v$ be the unique point in $T_0 \cap T_1$. Note that if $T_1=T$, then $T_0=\{v\}$ will simply be the trivial subtree containing one of the leaves of $T$.
	
	Suppose first that $v$ is not a leaf of $T_1$. Then there exist a decomposition $\{T',T''\}$ of $T_1$ such that $T',T''$ are non-trivial branches of $T_1$ with anchor $v$. If $|T'| \leq  l$ and $|T''| \leq  l$ then $T_1$ is $l$-minimal as desired. Thus we assume without loss of generality that $|T'| > l$. Then $\{T_0 \cup T'',T'\}$ is a decomposition of $T$ contradicting the choice of $\{T_0,T_1\}$.
	
	It remains to consider the case when $v$ is a leaf of $T_1$. Let $v'$ be chosen in the interior of the segment of $T_1$ containing $v$, so that $d_T(v,v')<l$. Then there exists a unique decomposition  $\{T',T''\}$  of $T_1$ such that   $T',T''$ are branches of $T_1$ with anchor $v'$. Without loss of generality assume that $v \in T''$. Then $T''$ is an interval with ends $v$ and $v'$ and so $|T''| < l$. If $|T'| \geq l$ then $\{T_0 \cup T'', T'\}$ again contradicts the choice of $\{T_0,T_1\}$. Otherwise, $T_1$ is $l$-minimal, as desired.
\end{proof}	

The next result immediately follows from \cref{l:partition1} by induction on the size of the decomposition.

\begin{lem}\label{l:partition2} For every tree $T$ and every $l>0$ there exist a decomposition $\{T_0,T_1,\ldots,T_k\}$ of $T$ such that $|T_0| \leq l$, and $T_1,\ldots,T_k$ are $l$-minimal.
\end{lem}

The following lemma is a variant of \cref{l:partition2} that will allow us to break a metric tree in pieces of roughly the same size.

\begin{lem}\label{l:partition3}
	For every metric tree $T$ and any $0 < l \leq |T|$ there exists a decomposition $\mathcal T$ of $T$ such that $l \leq |T'| \leq 3l$ for every $T'\in \mathcal T$.
\end{lem}

\begin{proof}
	Applying \cref{l:partition2} we obtain a decomposition $\mathcal T=\{T_0,\dots,T_k\}$ of $T$ such that $|T_0| \leq l$ and $T_1,\dots,T_k$ are $l$-minimal. By $l$-minimality, $l\leq |T_i| \leq 2l$ for every $i\in[k]$. Since $|T| \geq l$, if $k = 0$  then $|T_0| =l$ and $\mc{T}$ satisfies the lemma. If $k \geq 1$, since $T$ is connected, $T_0$ has to share a point with at least one other $T_i$, without loss of generality say $T_1$. Then, $l\leq |T_0\cup T_1| \leq 2l+l=3l$ and so $\{T_0\cup T_1,T_2,\dots,T_k\}$ satisfies the statement.
\end{proof}

\section{Covers}\label{s:covers}

In this section, we analyze the possible  (multi) sets of radii of balls that can be used to cover a given metric tree. These sets, which we simply call covers are formally defined as follows.

For a metric tree $T$,  $r \geq 0$ and $v \in T$, we denote by $B_T(v,r)$ the closed ball of radius $r$ with center $v$, i.e. 
$$B_T(v,r)=\{u \in T | d_T(v,u) \leq r\}.$$

A sequence $(r_1,\ldots,r_m)$ of non-negative reals is a \emph{cover} of $T$ if there exist $v_1,\ldots,v_m \in T$ such that $T = \cup_{i=1}^m B_T(v_i,r_i).$  

The equivalent of the following lemma, which precisely determines the minimum radius of a single ball that can cover a given metric tree, is well-known for discrete trees, but we include the proof for completeness.

\begin{lem}\label{l:diam}
	If $T$ is a tree, then there exists $y \in T$ such that $T= B_T\left(y,\frac{\diam(T)}{2}\right)$. In particular, $(r) \in \mc{C}(T)$ for every $r \geq \frac{\diam(T)}{2}.$
\end{lem}

\begin{proof}
	Let $u,v\in T$ such that $d_T(u,v)=\diam(T)$. Let $y$ be the point on the path $T[u,v]$ at distance exactly $\frac{\diam(T)}{2}$ from both $u,v$.
	
	Let $w\in T$. We claim that $d(w,y)\leq \frac{\diam(T)}{2}$. Suppose otherwise that $d(w,y)>\frac{\diam(T)}{2}$. Let $\mathcal T$ be the decomposition defined by breaking $T$ at $y$, in other words $\mathcal T$ is the set of minimal branches which anchor $y$. Let $T_w\in \mathcal T$ be such that $w\in T_w$, and analogously for $u,v$. Since $u,v$ are necessarily in different elements of $\mathcal T$, we have that, without loss of generality, $T_u\neq T_w$.
	
	Hence, the path $T[u,w]$ must contain $y$ and thus has length $d_T(u,y)+d_T(y,w)>\diam(T)$, which is a contradiction.
\end{proof}

One first application of \cref{l:diam} is the following.

\begin{lem}\label{l:largeeps}
	If $r_1,\ldots,r_k \in \bb{R}_+$ and $T$ is a metric tree such that $|T| \leq \sum_{i=1}^k r_i$, then $$(r_1,\ldots,r_k) \in \mc{C}(T).$$
\end{lem}
\begin{proof}
	By induction on $k$. 
	If $|T| \leq r_1$ then $(r_1)$ is a cover of $T$ by \cref{l:diam} and the lemma holds. Otherwise by \cref{l:partition1} there exists a decomposition $\{T_0,T_1\}$ of $T$ such that $T_1$ is $r_1$-minimal. Then $|T_1| \geq r_1$ and $(r_1)$ is a cover of $T_1$ by \cref{l:diam} since $\diam(T_1)\leq|T_1|\leq 2r_1$. Meanwhile,  $|T_0| \leq  \sum_{i=1}^{k-1} r_i$ and so $(r_2,\dots,r_k)$ is a cover of $T_0$ by the induction hypothesis. It follows that $(r_1,\ldots,r_k)$ is a cover of $T$, as desired.
\end{proof}

A similar (but discrete) lemma appears in \cite{bessy_bounds_2018} (a slightly stronger version appears in \cite{land_upper_2016}). Note that this last lemma already gives an interesting bound on the burning number of graphs. Consider the radii $\{0,\dots,k\}$ and a tree $T$ on $n$ vertices (which we will consider a metric tree). Then, if $|T|\leq\sum_{i=1}^{k} i=\frac{k(k+1)}{2}$, i.e. if $k\geq \sqrt{2n}+O(1)$, then $T$ can be burned (some of the centers of the balls might not be on vertices, so we might also need to increase each radii by 1, see the proof of \cref{lem:integralrounding} for more details). This approach to getting a bound of the form $b(G)\leq \sqrt{2n}+O(1)$ first appeared in \cite{bessy_bounds_2018} (it also appears in \cite{land_upper_2016,bonato_improved_2021}).

The following lemma is the main result of this section and is one of the key parts of our proof of \cref{thm:main}. It shows that every metric tree $T$ has covers of size two with the sum of the radii at most $\frac{|T|}{2}$, and a broad range of choices of radii. There is a certain technical trade-off here, where for certain trees the range is smaller but then so is the sum of radii. The precise  values of parameters matter here and are just right to make the following arguments work.

\begin{lem}\label{l:2cover} If $T$ is an $l$-minimal tree for some $l>0$, then there exists $0 \leq a \leq \min\{\frac{l}{2}-\frac{|T|}{4},\frac{|T|}{12}\}$ such that  \begin{itemize}
		\item $$\s{ \frac{|T|}{4}-3a-x
			, \frac{|T|}{4}+a+x
		}$$ is a cover of $T$ for all $0 \leq x \leq \frac{|T|}{4}-3a$, and 
		\item $$\s{ \frac{|T|}{4}+a+x
		}$$ is a cover of $T$ for 
		every $x \geq \frac{|T|}{4}-3a$.
	\end{itemize}
\end{lem}	

\begin{proof}
	Let $T'$ and $z$ be obtained by applying \cref{l:z} to $T$. Let $v_0,v_1,v_2,v_3 \in T'$ be chosen so that $L(T') \subseteq \{v_1,v_2,v_3\}$, $v_0$ is the unique point shared by the paths $T[v_1,v_2]$, $T[v_1,v_3]$ and $T[v_2,v_3]$.\footnote{More explicitly, if $T'$ has exactly three leaves we choose $v_0$ to be the unique branch point of $T'$ and $v_1,v_2,v_3$ to be the leaves. If $T$ has at most two leaves we choose $v_1$ and $v_2$ so that $L(T') \subseteq \{v_2,v_3\}$ and choose $v_0=v_1 \in T'$ arbitrarily.}  As
	$T'$ has at most three leaves, by condition (i) of \cref{l:z}, such a choice is always possible. Let $l_i = |T[v_0,v_i]|$ for $i=1,2,3$. Without loss of generality, suppose $l_1\leq l_2\leq l_3$.
	
	Then $|T'| = l_1 +l_2 +l_3$ and as $T'$ satisfies condition (iv) of \cref{l:z} we have\begin{equation}\label{e:T}|T| \geq l_1 +l_2 +l_3 + 4z. \end{equation} We have $\diam(T') =\max_{u,v \in L(T')}d_{T}(u,v)=d_{T}(v_2,v_3)=l_2+l_3$. From
	\cref{l:diamDecomp} (using condition (ii) of \cref{l:z}) it follows that $\diam(T) \leq l_2+l_3+2z$. 
	
	We show that $$a = \max\left\{0, \frac{l_1+l_2}{2} +z -\frac{|T|}{4}\right\}$$ satisfies the conditions of the lemma. 
	
	First, we need to verify that $a \leq\frac{|T|}{12}$ and $a \leq \frac{l}{2}-\frac{|T|}{4}$. The first of this inequalities follows immediately from \eqref{e:T}, as $$\frac{l_1+l_2}{2} +z \leq \frac{l_1+l_2+l_3+3z}{3} \leq \frac{|T|}{3},$$
	and so $a \leq \max\left\{0, \frac{|T|}{3} -\frac{|T|}{4}\right\} = \frac{|T|}{12}.$
	
	Showing that $a \leq \frac{l}{2}-\frac{|T|}{4}$ takes a bit more effort. 
	As $T$ is $l$-minimal there exists a decomposition $\{S_0,S_1\}$ of $T$ such that $|S_0| \leq l$ and $|S_1| \leq l$. In particular, $|T|=|S_0|+|S_1|\leq 2l,$ and so $\frac{l}{2}-\frac{|T|}{4} \geq 0$. Thus we only need to verify the case where  $a = \frac{l_1+l_2}{2} +z -\frac{|T|}{4} > 0$. We thus want to prove that $l\geq 2\left(a+\frac{|T|}{4}\right)=l_1+l_2+2z$. Substituting \eqref{e:T} in the above inequality $a>0$, we get  $l_1+l_2 > l_3$. In particular, $l_1 >0$ and so $T'$ has exactly three leaves $v_1,v_2$ and $v_3$. By condition (iii) of \cref{l:z}
	there exists a component $T'_i$ of $\bar{T}'$ with $v_i = \anc(T'_i)$ and $\dpt(T'_i)=z$ for every $i \in \{1,2,3\}$. Thus there exists $w_i \in T'_i$ such that $d_T(v_i,w_i)=z$.
	
	Without loss of  generality we assume that there exist distinct $i,j \in \{1,2,3\}$ such that $w_i,w_j \in S_0$. Then  \begin{align*}l  \geq |S_0| \geq d_T(w_i,w_j)=d_T(w_i,v_i)+d_T(v_i,v_j)+d_T(v_j,w_j)=z+(l_i+l_j)+z \geq l_1+l_2+2z,\end{align*}
	as desired.
	
	For $x \geq 0$, let $r_1=\frac{|T|}{4}+a+x$ and $r_2 = \frac{|T|}{4}-3a-x$ (note that since $a,x\geq 0$, we always have $r_1\geq r_2$). It remains to show that $(r_1)$ or $(r_1,r_2)$ is a cover of $T$.
	
	If $r_1 \geq \frac{l_2+l_3}{2}+z \geq \diam(T)/2$, then $(r_1)$ is a cover of $T$ by \cref{l:diam},  and so the lemma holds if  $x \geq  \frac{l_2+l_3}{2}+z - a - \frac{|T|}{4}$. Note that if we also have that $x\leq\frac{|T|}{4}-3a$, then we can say that $(r_1,r_2)$ is a cover, given that $(r_1)$ alone is a cover. This explains the different cutoff on $x$ between the statement of the theorem and in this proof.
	
	Thus we assume $x \leq  \frac{l_2+l_3}{2}+z - a - \frac{|T|}{4}$. In this regime we will show that $(r_1,r_2)$  is a cover of $T$. First, we have   
	\begin{align*} 
	r_2&=\frac{|T|}{4}-3a-x\\
	&\geq \frac{|T|}{2}-2a - \frac{l_2+l_3}{2}-z \\
	&= \min \left\{  \frac{|T|}{2} - \frac{l_2+l_3}{2}-z,  \frac{|T|}{2}- \left( l_1+l_2 +2z -\frac{|T|}{2} \right) - \frac{l_2+l_3}{2}-z \right\}  \\
	&\geq \min \left\{\frac{l_1+l_2+l_3+4z}{2}- \frac{l_2+l_3}{2}-z, l_1+l_2+l_3+4z-l_1-l_2 -2z  - \frac{l_2+l_3}{2}-z \right\}  \\&=\min \left\{ \frac{l_1}{2}+z,\frac{l_3-l_2}{2}+z \right\} \\&\geq z. 
	\end{align*}
	
	In particular, we are necessarily in the case $x\leq\frac{|T|}{4}-3a$, given that otherwise $r_2< 0\leq z$.
	
	Note that by our earlier assumption $r_2-z \leq r_1-z \leq \frac{l_2+l_3}{2}$. In particular, it is possible to choose $p_1,p_2 \in T[v_2,v_3]$ so that $d_T(p_1,v_2) = r_1-z$ and $d_T(p_2,v_3) = r_2-z$. We first wish to show that $T'\subseteq B_T(p_1,r_1-z) \cup B_T(p_2,r_2-z)$. Informally, we use the smaller radii to cover the end of longest branch of $T'$ and we use the larger  radii to cover the 2nd longest branch of $T'$ (including $v_0$), in order to, as we will see below, maximize the ``overflow'' onto the shortest branch of $T'$.
	
	First, we show that $T[v_2,v_3] \subseteq B_T(p_1,r_1-z) \cup B_T(p_2,r_2-z)$. To establish this it suffices to show that
	$d_T(p_1,p_2) \leq (r_1-z)+(r_2-z).$ By the above remark (that $r_2-z \leq r_1-z \leq \frac{l_2+l_3}{2}$), $p_1$ is closer to $v_2$ than $p_2$ is, and similarly $p_2$ is closer to $v_3$ than $p_1$ is (in other words, these points appear in the order $v_2-p_1-p_2-v_3$ on $T[v_2,v_3]$) and so $d_T(v_2,v_3)=l_2+l_3-(r_2-z)-(r_1-z)$. Therefore,
	\begin{align*}
	d_T(p_1,p_2) - \left((r_1-z)+(r_2-z)\right) &= l_2+l_3 -  2(r_2-z) - 2(r_1-z) \\
	&=
	l_2+l_3 - |T| + 4a +4z
	\\&=  \max \left\{l_2+l_3 -|T|+4z , l_2+l_3 - 2|T| + 2l_1+2l_2 +8z \right\} \\&\leq \max \left\{-l_1, l_2-l_3 \right\} \\&\leq 0,\end{align*}
	as desired.
	
	Next we show that $T[v_1,v_2] \subseteq  B_T(p_1,r_1-z).$ As we already have seen that $v_2 \in B_T(p_1,r_1-z)$, it suffices to show that $d_T(p_1,v_1) \leq r_1-z$. First note that $$2(r_1-z) \geq \frac{|T|}{2}+2a-2z \geq \frac{|T|}{2}+ \left(l_1+l_2+2z -\frac{|T|}{2}\right) -2z = l_1+l_2.$$
	Thus,
	\begin{align*}
		d_T(p_1,v_1)
		&=d_T(p_1,v_0)+d_T(v_0,v_1)\\
		&= |(r_1-z)-l_2| + l_1\\
		&= \max\{r_1-z - (l_2 -l_1), l_1+l_2-(r_1-z)\}\\
		&\leq r_1-z.
	\end{align*}
	
	It follows that $T' = T[v_2,v_3] \cup T[v_1,v_2] \subseteq  B_T(p_1,r_1-z) \cup B(p_2,r_2-z) $. Finally, we show that $T \subseteq  B_T(p_1,r_1) \cup B(p_2,r_2),$ i.e. we show that for every $u \in T$ we have $d_T(u,p_i) \leq r_i$ for some $i \in \{1,2\}.$ We already established this for $u \in T',$ so we may assume that $u \in \bar{T}'.$ Let $v$ be the anchor of the component $T''$ of $\bar{T}'$ containing $u$. Then $v \in T'$ and so $d_T(v,p_i) \leq r_i-z$ for some $i \in \{1,2\}$. Moreover, $d_T(u,v) \leq \dpt(T'') \leq z$, where the last inequality holds by the choice of $T'$ (condition (ii) of \cref{l:z}). Thus $d_T(u,p_i) \leq d_T(u,v) + d_T(v,p_i) \leq r_i$, as desired. 
\end{proof}

In \cite{bastide_improved_2022,omar_burning_2021}, it is used that if there are many leaves in a tree, we can cut them off, burn (obtain a cover of) the remaining subtree and then increment all radii by $1$ to burn the entire tree. Here, we have pushed this idea further by using \cref{l:z} to cut off as much as we need to obtain a subtree with at most $3$ remaining leaves.

\section{Random covers of metric trees}\label{s:random}

In this section we prove a fractional version of theorem \cref{thm:main} for metric trees. To state it we first need to formalize the notion of a random cover of a metric tree and define the necessary parameters of such a cover.

Let $T$ be a metric tree. Endow $\mc{C}(T)$ (the set of all covers of $T$) with the topology of $\sqcup_{m \in \bb{N}} \bb{R}^m_+$. Let $\nu$ be a finite  Borel measure on $\mc{C}(T)$. (Our main focus is the case when $\nu$ is a probability measure.)

The key parameter of interest to us
is the \emph{expectation measure $E\nu$} of $\nu$, which is  a Borel measure on $\bb{R}_+$ defined as follows. For $\bl{r}=(r_1,\ldots,r_m) \in \mc{C}(T)$ and $B \subseteq \bb{R}_+$, let $\#(B,\bl{r})$ denote the number of components  of $\bl{r}$ (i.e. radii) that lie in $B$, i.e. $$\#(B,\bl{r})=|\{ i \: \mid \: 1 \leq i \leq m, r_i \in B\}|.$$
Then, we can define $$E\nu(B)=\int \#(B,\bl{r}) d \nu(\bl{r})$$ for every Borel $B \subseteq \bb{R}_+$.\footnote{By rescaling we can assume that $\nu$ is a probability measure. Consider each $\bl{r}=(r_1,\ldots,r_m) \in \mc{C}(T)$ as a sum of discrete measures  $\sum_{i=1}^m\delta_{r_i}$. Then $\nu$ is a \emph{point process} on $\bb{R}_+$, and $E\nu$ is its \emph{expectation or intensity measure}. It is well-known that $E\nu$ is a well-defined and is indeed a measure, see e.g. \cite[Lemma 1.1.1]{reiss_course_1993}.}
In particular, when $\nu$ is a probability measure, then $E\nu(B)$ is the expected number of radii in a random cover $\bl{r}$ that lie in $B$. Note also that when $\nu$ is concentrated on covers of size $m$ then  $E\nu$ is the sum of $m$ marginals of $\nu$.

Before stating the main result of this section, we need to introduce a few more technical definitions. First, to convert the random covers in to a particular uniform one in the next section, it will be convenient to ensure that the covers we consider are somewhat tame. The precise notion of tameness is given in the next definition.   
For $l \in \bb{R}$, we say that a cover $\bl{r}=(r_1,\ldots,r_m)$ of $T$ is \emph{$l$-good} if $\|\bl{r}\|_1= \sum_{i=1}^m r_i \leq |T|+l$.
Let $\mc{C}(T,l) \subseteq \mc{C}(T)$ denote the set of $l$-good covers of $T$.

Secondly we will want the expectation measure of our distribution on covers to be close to uniform and as a result many of the calculation involve the uniform measures on real intervals.  
For $b  \geq a  \geq 0$, let $\bl{U}[a,b]$ denote the uniform probability (Borel) measure on $[a,b]$. For future reference, the following is useful identity relating the measures $\bl{U}[a,b]$, $\bl{U}[a,c]$ and $\bl{U}[b,c]$ for $c \geq  b \geq a$, $c > a$:
\begin{equation}\label{e:Urelation}
\bl{U}[a,c] = \frac{b-a}{c-a}\bl{U}[a,b] +  \frac{c-b}{c-a}\bl{U}[b,c].
\end{equation}

We are finally ready to state the main result of this section.

\begin{thm}\label{t:main2} If $\eps, r > 0$ and $T$ is a metric tree such that $|T| \geq 24 \eps^{-1}r$, then there exists a probability measure $\nu$ on $\mc{C}(T,r)$  such that $$E\nu \leq (1+\eps)\frac{|T|}{r}\bl{U}[0,r]. \qquad \footnote{For two  measures on the same measure space $\mu_1, \mu_2$ we write  $\mu_1 \leq \mu_2$ if $\mu_1(B) \leq \mu_2(B)$ for every measurable $B$.}$$ 			
\end{thm}

Informally, \cref{t:main2} implies that if $T$ is large enough compared to $r$ then there exists a distribution on ($r$-good) covers of $T$ that uses only radii in $[0,r]$, uses all such radii approximately equally often, and moreover the expected sum of the radii in a cover is not much larger than $\frac{|T|}{2}$.\footnote{The last property might not be obvious from the statement, but will be made clearer by subsequent calculations. Note moreover that if $T$ is an interval then the sum of the radii in every cover of $T$ is at least $\frac{|T|}{2}$ so this property and  the coefficient $(1+\eps)\frac{|T|}{r}$ in the theorem statement can not be improved, except for eliminating the $\eps$ error term.}

The rest of the section is occupied by the proof of \cref{t:main2} starting with introducing additional notation.

For a  Borel measure $\mu$ on $\bb{R}_{+}$, let $$\m(\mu)=2 \int x d\mu(x),$$ i.e. $\m(\mu)$ is the first moment of $\mu$ rescaled for convenience by a factor of two. If  $\nu$ is a probability measure on covers $\mc{C}(T)$ then 
$\m(E\nu)$ is twice the expected value of the sum of radii in a cover chosen according to $\nu$, i.e. the expected
maximum length of an interval that can be covered by such a cover. Due to this property we use $\m(E\nu)$ to keep track of the ``quality'' of the distribution $\nu.$ 

Note that $\m$ is a linear map from the  space of  Borel measures on $\bb{R}_{+}$ to $\bb{R}_+$ and that \begin{equation}\label{e:mU}\m(\bl{U}[a,b]) = b+a.\end{equation}

We prove of \cref{t:main2} iteratively for smaller and smaller $\eps$. A single iteration hinges on us finding a probability measure on $\mathcal C(T)$ that can be complemented by others with expectation measures of the form $c_i\bl{U}[0,a_i]$  for $a_i < r$ to produce the desired measure with expectation roughly uniform on the interval $[0,r]$.

The following definition makes the properties we need precise.
For $r, \delta > 0$, we say that a probability measure $\nu$ on $\mc{C}(T)$ is \emph{$(r,\delta)$-controlled} if there exist $\alpha_1,\ldots,\alpha_k \geq 0$ and $a_1,a_2,\ldots,a_k \in [0, (1-\delta)r]$  such that $$E\nu =\sum_{i=1}^{k}\alpha_i\bl{U}[a_i,r],$$

With the main definitions in place,
we collect all the basic properties of measures on $\mc{C}(T,l)$ that we need in the following lemma.

\begin{lem}\label{l:nuBasics} Let $T$ be a metric tree. 
	\begin{description}
		\item[(a)] Let $f_1,f_2,\ldots,f_m: [0,1] \to \bb{R}_+$ be affine functions such that $(f_1(x),f_2(x),\ldots,f_m(x))$ is a cover of $T$ for every $x \in [0,1]$.
		Let $a_i = \min\{f_i(0),f_i(1)\}$ and $b_i=\max\{f_i(0),f_i(1)\}$ for $i=1,\ldots,m$ and let $l=\max\{\sum_{i=1}^{m}f_i(0),\sum_{i=1}^{m}f_i(1)\}-|T|.$
		Then there exists a probability measure on $\nu$ on $\mc{C}(T,l)$ such that \begin{equation}\label{e:linear}E\nu=\sum_{i=1}^m\bl{U}[a_i,b_i].\end{equation}
		\item[(b)] Let $\{T_1,\ldots,T_k\}$ be a decomposition of a tree $T$. Let $r,\delta,l_1,\ldots,l_k \geq 0$ be such that for every $1 \leq i \leq k$ there exists an $(r,\delta)$-controlled  probability  measure  $\nu_i$ on $\mc{C}(T_i,l_i)$. Then there exists an $(r,\delta)$-controlled  probability  measure  $\nu$ on $\mc{C}(T,\sum_{i=1}^{k}l_i)$, such that
	\begin{equation}\label{e:basicB} \m(E\nu) = \sum_{i=1}^{k}\m(E\nu_i).\end{equation}
		\item[(c)] Let $l\geq 0$, let $\nu_0,\nu_1,\ldots,\nu_k$ be probability  measures on $C(T,l)$ and let $p_0,\ldots,p_k \geq 0$ be such that $\sum_{i=0}^k p_i=1$. Then there exists a probability measure $\nu$ on $C(T,l)$ such that
			\begin{equation}\label{e:basicC} E\nu=\sum_{i=0}^{k}p_i \cdot E\nu_i.\end{equation}
	\end{description}
	
\end{lem}
\begin{proof}
	\textbf{(a):} Let the map $F: [0,1] \to \mc{C}(T,l)$ be defined by 
	$$F(x)=(f_1(x),f_2(x),\ldots,f_m(x)).$$
	As $F$ is continuous it is Borel measurable and we can define
 $\nu = \bl{U}[0,1] \circ F^{-1}$ to be the image measure of the uniform probability measure on $[0,1]$ under the map $F$.\footnote{I.e. for every Borel $B \subseteq \bb{R}^m$ we have $\nu(B)=(\bl{U}[0,1])(F^{-1}(B))$} Then $\nu$ is a probability  measure on $\mc{C}(T,l)$ and the marginals $\nu_1,\ldots,\nu_i$ of $\nu$ satisfy $\nu_i= \bl{U}[0,1] \circ f_i^{-1}$. As the expectation measure of $\nu$ is the sum of its marginals we have \begin{equation}\label{e:linear0}E\nu =\sum_{i=1}^m \s{\bl{U}[0,1] \circ f^{-1}_i}.\end{equation}
As $f_i$ is an affine bijection from $[0,1]$ to $[a_i,b_i]$ we have $\bl{U}[0,1] \circ f^{-1}_i = \bl{U}[a_i,b_i]$, and so \eqref{e:linear0} implies \eqref{e:linear}, as desired.

\textbf{(b):} Let $\bl{r}^i=(r^i_1,\ldots,r^i_{m_i}) \in \mc{C}(T_i,l_i)$ for $i=1,\ldots,k$. Define \begin{align*}
\bl{R}&=\bl{R}(\bl{r}^1,\ldots,\bl{r}^k) \\&= (r^1_1,\ldots,r^1_{m_1},r^2_1,\ldots,r^2_{m_2},\ldots,r^k_1,\ldots,r^k_{m_k})
\end{align*}
to be the concatenation of these covers. Then $\bl{R}$ is a cover of $T$ as there exist balls of radii
$r^i_1,\ldots,r^i_{m_i}$ whose union includes $T_i$ for every $i$, and so the union of such balls over all $i$ is $T$. Note also that $$\|\bl{R}\|_1 = \sum_{i=1}^{k}\|\bl{r}^i\|_1 \leq  \sum_{i=1}^{k}(|T_i|+l_i) = |T|+  \sum_{i=1}^{k}l_i,$$
and so $\bl{R}$ is $(\sum_{i=1}^{k}l_i)$-good.

Thus
$\bl{R}$ is a Borel measurable map from $\prod_{i=1}^k \mc{C}(T_i,l_i)$ to   $\mc{C}(T,\sum_{i=1}^{k}l_i)$, and we define $\nu =  (\otimes_{i=1}^k\nu_i) \circ \bl{R}^{-1}$. That is, $\nu$ is the probability measure on $(\sum_{i=1}^{k}l_i)$-good covers of $T$ obtained by  taking the union (more formally, a concatenation) of $l_i$-good covers of $T_i$  chosen for $i=1,\ldots,k$ independently at random according to the probability  measure $\nu_i$.
Then $E\nu = \sum_{i=1}^kE\nu_i$, implying that $\nu$ is $(r,\delta)$-controlled as each $\nu_i$ is, and implying \eqref{e:basicB} by linearity of $\m(\cdot)$.
 
\textbf{(c):} Let $\nu= \sum_{i=0}^{k}p_i \nu_i$. That is $\nu$ the probability measure on $l$-good covers of $T$ obtained by randomly choosing an index $\{0,\dots,k\}$ with $i$ chosen with probability $p_i$ and then choosing a 
cover of $T$ according to the probability  measure $\nu_i$.  The identity \eqref{e:basicB} holds as the expectation measure is linear.
\end{proof}

Our next two lemmas establishes existence of a measure that is needed to perform a single iteration in the proof of \cref{t:main2} as outlined above. The first lemma finds a measure on covers of a single part of an appropriate decomposition of $T$, and the second combines the measures for each part of the decomposition.
 
\begin{lem}\label{l:controlled1} If $0 < r$, $0 < \delta \leq \frac{1}{2}$ are real and $T$ is a $2(1-\delta)r$-minimal metric tree, then there exists an  $(r,\delta)$-controlled probability measure $\nu$ on $\mc{C}(T,0)$ such that $$\m(E\nu) \leq |T| + \delta r.$$
\end{lem}

\begin{proof}  
	Let $l=2(1-\delta)r$. Let $0 \leq a \leq \frac{l}{2}-\frac{|T|}{4}$ be as in \cref{l:2cover}. Then $|T| \geq l \geq \frac{|T|}{2}+2a$.
	
	Suppose first that $\frac{|T|}{2}-2a \leq (1-\delta)r$. By \cref{l:2cover} we have that $$\s{ \frac{|T|}{4}+a+x
		}$$ is a cover of $T$ for 
		every $x \geq \frac{|T|}{4}-3a$. In other words, $(x')$ is a cover of $T$ for every $x'\geq \frac{|T|}{4}+a+\left(\frac{|T|}{4}-3a\right)=\frac{|T|}{2}-2a$. In particular, this holds if $x' \in [(1-\delta)r,r]$ by our previous assumption. Hence, by \cref{l:nuBasics}(a) there exists a probability measure $\nu$ on $\mc{C}(T,r-|T|)$ such that
	$$E\nu = \bl{U}[(1-\delta)r,r].$$ Given that $|T|\geq 2(1-\delta)r\geq r$ we have that $\mc{C}(T,r-|T|)\subseteq \mc{C}(T,0)$ and so $\nu$ is a probability distribution on $\mc{C}(T,0)$.
	
	It is direct from the definition that $\nu$ is $(r,\delta)$-controlled (take $k=1,\alpha_1=1,a_1=(1-\delta)r$).  
	Using \eqref{e:mU}, we have $$\m(E\nu) =  (2-\delta)r = l +\delta r \leq |T|+ \delta r.$$
	It follows that  $\nu$ satisfies the conditions of the lemma.
	
	Thus we assume $\frac{|T|}{2}-2a \geq (1-\delta)r$. For $y \in [0,1]$, \begin{equation}\label{e:cover0}\s{\max\left\{0,\frac{|T|}{4}-3a-y\s{r-\frac{|T|}{4}-a}\right\},\frac{|T|}{4}+a+y\s{r-\frac{|T|}{4}-a}}\end{equation} is a cover of $T$ by \cref{l:2cover} applied with $x = y\s{r-\frac{|T|}{4}-a}$. Note that as $r \geq \frac{l}{2} \geq \frac{|T|}{4} + a$, we indeed have $x \geq 0$.
	
	Suppose further that $\frac{|T|}{2}-2a \geq r$.
	By increasing the first component , we further deduce that 
	 $$\s{\frac{|T|}{4}+a-y\s{r-\frac{|T|}{4}+3a},\frac{|T|}{4}+a+y\s{r-\frac{|T|}{4}-a}}$$
	 is a cover of $T$ for every such $y$. That this is actually an increase of the first radii is a consequence of the last supposition (to show that the new first radii is non-negative) and of the fact that $y\leq 1$ (for showing the inequality in the second case). 
	 
   	By \cref{l:nuBasics}(a) there exists a probability distribution $\nu$ on $\mc{C}\left(T,\frac{|T|}{2}+2a - |T|\right)$ such that
\begin{equation}\label{e:cover1} E\nu =  \bl{U}\left[\frac{|T|}{2}-r-2a, \frac{|T|}{4}+a\right]+\bl{U}\left[\frac{|T|}{4}+a,r\right].\end{equation}
As $\frac{|T|}{2}+2a \leq  l \leq |T|$, $\mc{C}\left(T,\frac{|T|}{2}+2a - |T|\right) \subseteq \mc{C}(T,0) $ and so $\nu$ is a probability distribution on $\mc{C}(T,0) $. 
Using \eqref{e:Urelation} we have $$\bl{U}\left[\frac{|T|}{2}-r-2a,r\right] = \frac{r -\frac{|T|}{4}+3a}{2r- \frac{|T|}{2}+2a}\bl{U}\left[\frac{|T|}{2}-r-2a, \frac{|T|}{4}+a\right] + \frac{r -\frac{|T|}{4}-a}{2r- \frac{|T|}{2}+2a}\bl{U}\left[\frac{|T|}{4}+a,r\right],$$ so we can rewrite \eqref{e:cover1} as
$$ E\nu  =  \frac{2r- \frac{|T|}{2}+2a}{r -\frac{|T|}{4}+3a}\bl{U}\left[\frac{|T|}{2}-r-2a,r\right] + \frac{4a}{r -\frac{|T|}{4}+3a}\bl{U}\left[\frac{|T|}{4}+a,r\right]. $$
As $0  \leq \frac{|T|}{2}-r-2a \leq \frac{|T|}{4}+a \leq \frac{l}{2} =  (1-\delta)r$, it follows that $\nu$ is $(r,\delta)$-controlled. 
Using \eqref{e:mU} and \eqref{e:cover1}, we have 
	$$\m(E\nu) = \s{\s{\frac{|T|}{2}-r-2a}+\s{\frac{|T|}{4}+a}}+\s{\s{\frac{|T|}{4}+a}+r}  = |T|.$$
Thus $\nu$ satisfies the lemma in this case.
	
	It remains to consider the case  $(1-\delta)r \geq \frac{|T|}{2}-2a \leq r$. By increasing the first component in the cover \eqref{e:cover0} in a different way, we see that $$\s{\frac{|T|}{4}+a - y\s{\frac{|T|}{4}+a},\frac{|T|}{4}+a+y\s{r-\frac{|T|}{4}-a}}$$
	is a cover of $T$ for every $y \in [0,1]$. Verifying that the new first radii is non-negative is direct, and furthermore this new radii is indeed an increase in the other case as a consequence of $\frac{|T|}{2}-2a \leq r$.
	As \begin{align*}\max_{y \in \{0,1\}}&\s{\frac{|T|}{4}+a - y\s{\frac{|T|}{4}+a}+\frac{|T|}{4}+a+y\s{r-\frac{|T|}{4}-a}} = \max\left
	\{r, \frac{|T|}{2}+2a\right\} \leq l \leq |T|,
\end{align*}
	by \cref{l:nuBasics}(a) there exists a probability distribution $\nu$ on $\mc{C}(T,0)$ such that
	\begin{equation}\label{e:cover2}E\nu =  \bl{U}{\left[0, \frac{|T|}{4}+a\right]}+\bl{U}{\left[\frac{|T|}{4}+a,r\right]}.\end{equation}
	As in the previous case, we can use \eqref{e:Urelation} to get
	$$\bl{U}\left[0,r\right] = \frac{\frac{|T|}{4}+a}{r}\bl{U}\left[0,\frac{|T|}{4}+a \right] + \frac{r-\frac{|T|}{4}-a}{r}\bl{U}\left[\frac{|T|}{4}+a ,r\right],$$
	we rewrite \eqref{e:cover2} as  $$E\nu = \frac{r}{\frac{|T|}{4}+a}\bl{U}{[0, r]}+ \frac{\frac{|T|}{2}+2a-r}{\frac{|T|}{4}+a}\bl{U}{\left[\frac{|T|}{4}+a,r\right]}.$$
	As observed earlier $\frac{|T|}{4}+a \leq  (1-\delta)r$ and so $\nu$ is $(r,\delta)$-controlled. Finally, using \eqref{e:mU} and \eqref{e:cover2}, we have 
$$\m(E\nu) = r + \frac{|T|}{2}+2a  \leq |T| + \delta r$$
	 and $\nu$ satisfies the lemma in this last case. 
\end{proof}

\begin{lem}\label{l:controlled2}
	For every $0< \delta \leq 1/2$, every $r > 0$ and every metric tree $T$ there exists an $(r,\delta)$-controlled  probability  measure  $\nu$ on $\mc{C}(T,r)$ such that
$$\m(E\nu)  \leq (1+\delta)|T| + 2r.$$
\end{lem}

\begin{proof}
	Let $l=2(1-\delta)r \geq r$. By \cref{l:partition2} there exists a decomposition $\{T_0,T_1,\ldots,T_k\}$ of $T$ such that $|T_0| \leq l$, and $T_1,\ldots,T_k$ are $l$-minimal. Note that
	$$|T| \geq \sum_{i=1}^k |T_i| \geq kl \geq kr. $$
	
	By \cref{l:controlled1}, for every $i \in [k]$  there exists an $(r,\delta)$-controlled distribution $\nu_i$ on $\mc{C}(T_i,0)$ such that $\m(E\nu_i) \leq |T_i|+\delta r$.
	
	Meanwhile, by \cref{l:diam}, $\diam(T_0) \leq \frac{|T_0|}{2} \leq (1-\delta)r$ and so $(x)$ is a cover of $T_0$ for every $x\geq (1-\delta)r$, and in particular for $x\in [(1-\delta)r,r]$. Applying \cref{l:nuBasics}(a), there exists a probability measure $\nu_0$ on $\mc{C}(T_0,r-|T_0|)\subseteq \mc{C}(T_0,r)$ such that $E\nu_0 = \bl{U}[(1-\delta)r,r]$, and so $\nu_0$ is $(r,\delta)$-controlled and $\m(E\nu_0) \leq (2-\delta)r$.
	
	Then, by \cref{l:nuBasics}(b) there exists an $(r,\delta)$-controlled  probability  measure  $\nu$ on $\mc{C}(T,r)$,
	such that $$\m(E\nu) \leq  \sum_{i=1}^k |T_i| + k\delta r  + (2-\delta)r \leq |T| +  k\delta r  + 2r \leq (1+\delta)|T| + 2r,$$
	as desired.
\end{proof}

With all the ingredients in place, we start the proof of our main result, which we restate for convenienve.

\Main*

\begin{proof}[Proof of~\cref{t:main2}.]
	First we show that the theorem holds for $\eps \geq 12/11$. Let $k = \lceil \frac{|T|}{r}\rceil \leq \frac{|T|}{r}+1$. By \cref{l:largeeps} $$\s{\underbrace{x,\ldots,x}_{k \text{ times}}, \:\underbrace{r-x,\ldots,r-x}_{k \text{ times}}} \in \mc{C}(T)$$ for every $x \in [0,r]$, since the sum of these radii is $kr\geq |T|$. By \cref{l:nuBasics} there exists a
	 probabilistic distribution $\nu$ on $\mc{C}(T,kr-|T|)\subseteq \mc{C}(T,r)$ such that $$E\nu = 2k\cdot\bl{U}[0,r] \leq  2\s{\frac{|T|}{r}+1}\bl{U}[0,r] \leq 2\s{1+\frac{\eps}{24}}\frac{|T|}{r}\bl{U}[0,r]\leq  (1+\eps)\frac{|T|}{r}\bl{U}[0,r],$$
	 as desired, where the last inequality holds by the choice of $\eps$ to be relatively large.
	 
	 We now prove that the theorem holds for $\eps$ such that $\eps \geq \frac{15}{n}$ for some integer $n$ by induction on $n$, which implies that the theorem holds all $\eps>0$. The base case for $n \leq 13$ was established above.
	 
	 Suppose now $\frac{15}{n-1} > \eps \geq \frac{15}{n}$ for $n \geq 14$.	Let $\eps' =\eps+\frac{\eps^2}{13}, \delta=\frac{\eps}{8}$. Then $$\eps' \geq \frac{15}{n}\s{1+\frac{15}{13n}
	 } \geq \frac{15}{n-1}$$
	by our lower bound on $n$.
 	Thus the theorem holds for $\eps'$ by the induction hypothesis.

	 By \cref{l:controlled2} there exists a probability distribution $\nu_0$ on $\mc{C}(T,r)$, as well as reals $\alpha_1,\ldots,\alpha_k \geq 0$ and $0 \leq a_1,a_2,\ldots,a_k \leq (1-\delta)r$, such that $$E\nu_0 =\sum_{i=1}^{k}\alpha_i\bl{U}[a_i,r],$$ and \begin{equation}\label{e:mn0} \m(E\nu_0) \leq (1+\delta)|T| + 2r \leq \s{1+\frac{\eps}{8}}|T|+\frac{\eps}{12}|T| \leq  \s{1+\frac{\eps}{4}}|T|, \end{equation}
 where the second to last inequality uses the choice of $\delta$ and the condition $r \leq \frac{\eps|T|}{24}$ in the theorem statement.
 
	 By \eqref{e:mU} and the linearity of $\m(\cdot)$ we have $\m(E\nu_0) = \sum_{i=1}^k\alpha_i(r+ a_i)$, and so \eqref{e:mn0} implies 
	 \begin{equation}\label{e:mn01} \sum_{i=1}^k\alpha_i(r+ a_i)\leq \s{1+\frac{\eps}{4}}|T|.\end{equation}

	  Note that $|T|\geq \frac{24r}{\varepsilon}\geq\frac{24a_i}{\varepsilon'}$ for every $i\in [k]$. Then by the induction hypothesis, for each $i\in [k]$ there exists a probabilistic distribution $\nu_i$ on $\mc{C}(T,a_i) \subseteq \mc{C}(T,r)$ such that $E\nu_i \leq(1+\eps')\frac{|T|}{a_i}\bl{U}{[0,a_i]}$. Let $$q = (1+\eps')|T|+\sum_{i=1}^k\frac{\alpha_ia_i^2}{r-a_i},  \qquad p_0 = \frac{(1+\eps')|T|}{q}$$ and let $$p_i = \frac{\alpha_i a_i^2}{(r-a_i)q}$$ for $i \in [k]$. Of course, $\sum_{i=0}^k p_i=1$. 
	 By \cref{l:nuBasics}(c) here exists is a probability measure $\nu$ on $\mc{C}(T,r)$ such that
	 \begin{align*}
	 E\nu  
	 &\leq p_0 \cdot E\nu_0 +\sum_{i=1}^{k} \s{p_i \cdot E\nu_i}\\
	 &=  p_0\sum_{i=1}^{k} \s{\frac{p_i}{p_0} \cdot E\nu_i +  \alpha_i\bl{U}[a_i,r]} \\ 
	 &\leq p_0\sum_{i=1}^{k}  \s{\frac{\alpha_i a_i^2}{ (1+\eps')|T|(r-a_i)} \cdot (1+\eps')\frac{|T|}{a_i}\bl{U}{[0,a_i]} + \alpha_i\bl{U}[a_i,r]} \\ 
	 &=
	 p_0\sum_{i=1}^{k} \s{ \frac{\alpha_i r}{(r-a_i)}\s{\frac{a_i}{r}\bl{U}{[0,a_i]} + \frac{r-a_i}{r}\bl{U}[a_i,r]}} \\ 
	 &= p_0\s{\sum_{i=1}^{k}\frac{\alpha_i r}{(r-a_i)}}\bl{U}[0,r],
	  \end{align*}
	  using in particular \eqref{e:Urelation}.

Thus	it suffices to show that
\begin{equation}\label{e:ai0} \frac{(1+\eps)|T|}{r} \geq p_0\s{\sum_{i=1}^{k}\frac{\alpha_i r}{(r-a_i)}}.\end{equation}
Substituting the value of $p_0$, expanding $q$ and rearranging, we obtain that inequality \eqref{e:ai0} is equivalent to
\begin{align*}
\sum_{i=1}^{k}\frac{\alpha_i r^2}{r-a_i} \leq 
(1+\eps)|T| + \frac{1+\eps}{1+\eps'}\sum_{i=1}^k\frac{\alpha_i a_i^2}{r-a_i}, \end{align*}
which we can also rewrite as
\begin{align}\label{e:ai}
 \sum_{i=1}^{k}\alpha_i (r+a_i) + \frac{\eps'-\eps}{1+\eps'}\sum_{i=1}^k\frac{\alpha_i a_i^2}{r-a_i} \leq 
 (1+\eps)|T|.
\end{align}
As $a_i \leq (1-\delta) r$ for every $i\in [k]$ we have $$\frac{a_i^2}{r-a_i} \leq \frac{r(r+a_i)}{2(r-a_i)} \leq \frac{(r+a_i)}{2\delta},$$ and therefore the second term of the left-hand side of \eqref{e:ai} is upper bounded by $$\frac{\eps'-\eps}{2\delta}\sum_{i=1}^{k}\alpha_i (r+a_i)= \frac{4\eps}{13}\sum_{i=1}^{k}\alpha_i (r+a_i),$$
where the equality holds by the choice of $\delta$ and $\eps'$.
Thus \eqref{e:ai} is implied by
\begin{equation}\label{e:ai1} \s{1+ \frac{4\eps}{13}}\sum_{i=1}^{k}\alpha_i (r+a_i)  \leq (1+\eps)|T|. \end{equation}
By \eqref{e:mn01}, the inequality \eqref{e:ai1} is further implied by
$$  \s{1+ \frac{4\eps}{13}} \s{1+\frac{\eps}{4}}|T| \leq (1+\eps)|T|.$$
This last inequality can finally easily seen to hold for $\eps \leq 2$.
\end{proof}

\section{Proof of \cref{thm:main}}\label{s:proof}

In this section, we will deduce our main result from \cref{t:main2}. We will need some classical results in probability theory. 

\begin{thm}[Markov's inequality]\label{thm:markov}
	If $X\geq 0$ is a random variable and $a>0$, then 
	$$\P(X\geq a)\leq \frac{\E[X]}{a}.$$
\end{thm}

\begin{thm}[Hoeffding's inequality]\cite{hoeffding_probability_1963}\label{thm:hoeffding}
	If $X_1,\dots, X_m$ are independent random variables with values in $[a,b]$, $X=\sum_{i=1}^m X_i$ and $t>0$, then$$\P\left(X\geq \E[X] + t\right)\leq\exp\left(-\frac{2t^2}{m(b-a)^2}\right).$$
\end{thm}

We now convert the random covers given by \cref{t:main2} into uniform covers, proving a metric (but not fractional) equivalent of \cref{thm:main}. 

\begin{thm}\label{thm:integerversion}
	For every $\varepsilon>0$, there exists $K=K_{\ref{thm:integerversion}}(\varepsilon)$ such that if $K\leq k\in \N$ and $T$ is a metric tree such that $|T|\leq (1-\varepsilon)k^2$, then $(1,2,\ldots,k)$ is a cover of $T$.
\end{thm}

\begin{proof}
	We may of course assume that $\varepsilon<1$. Choose $N\in \N$ large enough so that $\lambda:=1-\frac{\left(1+\frac{1}{N}\right)(1-\varepsilon)}{\left(1-\frac{1}{N}\right)^2}>0$. Choose $P\geq 1$ large enough so that $N\exp\left(-\frac{\lambda^2}{24N^4(1-\varepsilon)^2}P\right) < \lambda$. Let $D=24N$ and  $K = \frac{PD}{1-\varepsilon}$. We show that $K$ satisfies the theorem.
	
	We assume without loss of generality $|T| = (1-\varepsilon)k^2 \geq PDk$ by extending $T$ if needed.
	
By \cref{l:partition3},  there exist a decomposition $\{T_1,\dots,T_m\}$ of $T$ such that $Dk\leq|T_i|\leq 3Dk$ for every $i\in \{1,\dots,m\}$. Note that this implies that $mDk\leq |T|\leq 3mDk$ and in particular $m\geq \frac{P}{3}$.
	
	Setting $r=\left(1-\frac{1}{N}\right)k$, apply \cref{t:main2} to $T_i$ for each $i\in \{1,\dots,m\}$ with the parameter $\eps=\frac{24r}{|T_i|}$to obtain a probability measure $\nu^i$ on $\mathcal C(T_i,r)$ such that $$E\nu^i \leq\left(1+\frac{24r}{|T_i|}\right)\frac{|T_i|}{r}\bl{U}{[0,r]}\leq \left(1+\frac{1}{N}\right)\frac{|T_i|}{r}\bl{U}{[0,r]}.$$ 
	
	Let $\bl{r}^i$ be a random cover of $T_i$ following the law $\nu^i$. For $0\leq \alpha\leq 1-\frac{2}{N}$, let $$Z_i(\alpha) = \#\s{\left[\alpha k, \s{\alpha+\frac{1}{N}}k \right],\bl{r}^i}$$ be the random variable equal  to the number of radii in $\bl{r}_i$ that belong to the interval $\left[\alpha k, (\alpha+\frac{1}{N})k\right]$.
Then $\E\left[Z_i(\alpha)\right] = E\nu^i\left(\left[\alpha k, (\alpha+\frac{1}{N})k\right]\right)$ by definition of $E\nu^i$ and so
	$$\E\left[Z_i(\alpha)\right]\leq \left(1+\frac{1}{N}\right)\frac{|T_i|}{r}\cdot \left(\bl{U}[0,r]\right)\left(\left[\alpha k,\left(\alpha+\frac{1}{N}\right)k\right]\right)=\frac{\left(1+\frac{1}{N}\right)|T_i|}{N\left(1-\frac{1}{N}\right)^2k}$$
	for every $1 \leq i \leq m$.
	Hence, $$\E\left[\sum_{i=1}^m Z_i(\alpha)\right]\leq \frac{\left(1+\frac{1}{N}\right)|T|}{N\left(1-\frac{1}{N}\right)^2k}= \frac{\left(1+\frac{1}{N}\right)(1-\varepsilon)k^2}{N\left(1-\frac{1}{N}\right)^2k}=\frac{(1-\lambda)k}{N}.$$
	
	Firstly, using Markov's inequality (\cref{thm:markov}),
	$$\P\left(\sum_{i=1}^m Z_i(0)\geq \frac{k}{N}\right)\leq \frac{\E\left[\sum_{i=1}^m Z_i(0)\right]}{\frac{k}{N}}\leq 1-\lambda.$$
	
	This above inequality holds for any $\alpha$, but we will need a stronger result in general. When $\alpha\geq \frac{1}{N}$, we will use a Chernoff-type bound.
	
	Since $\bl{r}^i$ is $r$-good, for every $i\in \{1,\dots,m\}$ we have
	$$Z_i(\alpha) \cdot \alpha k \leq \|\bl{r}^i\|_1  \leq |T_i|+r\leq (3D+1)k\leq 4Dk$$
	and thus 
	$$0\leq Z_i(\alpha)\leq \frac{4D}{\alpha}\leq \frac{4N|T|}{mk}=\frac{4N(1-\varepsilon)k}{m}.$$
	using the conditions $|T|\geq mDk$ and $|T|=(1-\varepsilon)k^2$.
	
	Then, applying Hoeffding's inequality (\cref{thm:hoeffding}), we get
	\begin{align*}
		\P\left(\sum_{i=1}^m Z_i(\alpha)\geq \frac{k}{N}\right)
		&\leq \P\left(\sum_{i=1}^m Z_i(\alpha)\geq \E\left[\sum_{i=1}^m Z_i(\alpha)\right]+\left(\frac{k}{N}-\frac{(1-\lambda)k}{N}\right)\right)\\
		&\leq\exp\left(-\frac{2\left(\frac{\lambda k}{N}\right)^2}{m\cdot \left(\frac{4N(1-\varepsilon)k}{m}\right)^2}\right)\\
		&=\exp\left(-\frac{\lambda^2}{8N^4(1-\varepsilon)^2}m\right)\\
		&\leq\exp\left(-\frac{\lambda^2}{24N^4(1-\varepsilon)^2}P\right).
	\end{align*}

	Combining these, by the union bound we have
	\begin{align*}
		\P\left(\exists j\in \{0,\dots,N-2\}, \sum_{i=1}^m Z_i\left(\frac{j}{N}\right)\geq \frac{k}{N}\right)
		&\leq \sum_{j=0}^{N-2} \P\left(\sum_{i=1}^m Z_i\left(\frac{j}{N}\right)\geq \frac{k}{N}\right)\\
		&<(1-\lambda)+N\exp\left(-\frac{\lambda^2}{24N^4(1-\varepsilon)^2}P\right)\\
		&<1
	\end{align*}
	by our choice of $P$. Hence, there exist covers $\bl{r}^1, \ldots, \bl{r}^m$  of $T_1,\dots,T_m$, respectively using radii in $[0,(1-\frac{1}{N})k]=[0,r]$ such that the total number of radii in $\left[\frac{j}{N}k,\frac{j+1}{N}k\right]$ that are used is smaller than $\frac{k}{N}$, and so at most $\left\lfloor\frac{k}{N}\right\rfloor$.
	
	The concatenation of these covers is a cover $\bl{r}$ of $T$, from which
	we now construct a cover using radii $\{0,\dots,k\}$. Let $j\in\{1,\dots,N-2\}$. The interval $\left(\frac{j+1}{N}k,\frac{j+2}{N}k\right]$ contains at least $\left\lfloor\frac{k}{N}\right\rfloor$ integers. Thus we can increase the radii in  $\bl{r}$ that lie in the interval $\left[\frac{j}{N}k,\frac{j+1}{N}k\right]$ replacing them by distinct integers in $\left(\frac{j+1}{N}k,\frac{j+2}{N}k\right]$.
	The radii in the resulting modified cover are distinct integers in the interval $[\frac{1}{N}k, k]$, implying that $(1,2,\ldots,k)$ is a cover of $T$, as desired.
\end{proof}

We now  come back to the original discrete setting. Let $T$ be a (discrete) tree. As in the metric setting, for  $r \geq 0$ and $v \in V(T)$, we denote by $B_T(v,r)$ the closed ball of radius $r$ with center $v$, i.e. 
$$B_T(v,r)=\{u \in V(T) | d_T(v,u) \leq r\},$$
where $ d_T(\cdot,\cdot)$ is the usual graph metric, i.e. $d_T(v,u)$ is the number of edges in the unique path with ends $u$ and $v$. 
As before, $(r_1,\ldots,r_m)$ is a \emph{cover} of $T$ if there exist $v_1,\ldots,v_m \in V(T)$ such that $V(T) = \cup_{i=1}^m B_T(v_i,r_i).$ We denote by $T^{M}$  a metric tree obtained from $T$ by replacing each edge by an interval of length one. Note that $d_{T^M}(u,v) = d_T(u,v)$  for any $u,v \in V(T)$.

\begin{lem}\label{lem:integralrounding}
	If $T$ be a discrete tree and $(r_1,\ldots,r_m)$ is a cover of the corresponding metric tree $T^M$, then $(r_1+1,\ldots,r_m+1)$ is a cover of $T$.
\end{lem}
\begin{proof}
	 Let $v_1,\ldots,v_m \in T^M$ be such that $T^M = \cup_{i=1}^m B_{T^M}(v_i,r_i).$

	For each $i\in [m]$, let $u_i\in V(T)$ be an end point of the segment of $T^M$ containing $v_i$.  In particular $d_{T^M}(u_i,v_i) \leq 1$, and so $B_{T^M}(v_i,r_i) \subseteq B_{T^M}(u_i,r_i+1)$, implying $T^M = \cup_{i=1}^m B_{T^M}(u_i,r_i+1).$  
	 As $V(T) \subseteq T^M$ and the distances between vertices of $T$ are preserved in $T^M$, we have $V(T) = \cup_{i=1}^m B_{T}(u_i,r_i+1),$  
implying that $(r_1+1,\ldots,r_m+1)$ is a cover of $T$.
\end{proof}

Our main result readily follows from \cref{thm:integerversion} and \cref{lem:integralrounding}.

\begin{proof}[Proof of \cref{thm:main}] We need to show that for every $0 < \varepsilon < 1$, there exists $N$ such that if $G$ is a connected graph on $n\geq N$ vertices then $b(G)\leq (1+\varepsilon)\sqrt{n}$. 
	
	Let $\eps'=\frac{\eps}{2}$, we show that  $$N = \max\left\{\left(K_{\ref{thm:integerversion}}(\eps')\right)^2,\frac{6}{\eps}\right\}$$ has the above property.

	Let $G$ be  a connected graph on $n\geq N$ vertices. As noted in the introduction, it suffices to consider any spanning tree $T$ of $G$; burning $T$ will also burn $G$, hence $b(G)\leq b(T)$. Let $T^M$ be the metric tree corresponding to $T$.
	
	Set $k=\left\lceil\sqrt{\frac{n}{1-\varepsilon'}}\right\rceil$. We have that $k\geq \sqrt{\frac{n}{1-\varepsilon'}}\geq \sqrt{\frac{1}{1-\varepsilon'}N}\geq K_{\ref{thm:integerversion}}(\varepsilon')$ and $|T^M|=n-1\leq (1-\varepsilon')\left\lceil\sqrt{\frac{n}{1-\varepsilon'}}\right\rceil^2= (1-\varepsilon')k^2$.
	
	Hence, by \cref{thm:integerversion}, $(1,\ldots,k)$ is a cover of $T^M$. By \cref{lem:integralrounding} it follows that $(2,\ldots,k+1)$  is a cover of $T$. Thus $$b(T)\leq (k+1)+1=\left\lceil\sqrt{\frac{n}{1-\frac{\varepsilon}{2}}}\right\rceil+2\leq \left(1+\frac{\varepsilon}{2}\right)\sqrt{n}+3\leq (1+\varepsilon)\sqrt{n},$$ as desired, where the last inequality uses the condition $n \geq N \geq \frac{6}{\eps}.$
\end{proof}

\section{Concluding remarks}\label{s:remark}

\subsection{Eliminating the error}

We have shown that the Burning Number Conjecture holds asymptotically, i.e. $b(G) \leq (1+o(1))\sqrt{n}$ for every connected $n$ vertex graph $G$. A natural next direction would be to attempt to eliminate the error term, proving the conjecture in full.

Unfortunately, our method does not seem to give much insight in the behaviour of burning number of small graphs. It might be conceivable that our argument can be used as a starting point for the proof of the Burning Number Conjecture for sufficiently large $n$, but such an extension is likely to be quite difficult and require additional ideas.

On the other hand, eliminating the error term in \cref{t:main2} is more likely to be within reach. The conclusion of  \cref{t:main2} does not hold with $\eps=0$ if $T$ is an interval with $|T| < 2r$, as the radii of length $\frac{|T|}{2}$ can not be utilized without waste, but we conjecture that this is the only obstruction. 

\begin{conjecture}\label{conj:main2} If $ r > 0$ and $T$ is a metric tree such that $|T| \geq 2r$, then there exists a probability measure $\nu$ on $\mc{C}(T,0)$  such that $$E\nu \leq \frac{|T|}{r}\bl{U}[0,r].$$ 			
\end{conjecture}

It suffices to prove \cref{conj:main2} for metric trees $T$ that do not admit a decomposition $\{T_1,T_2\}$ such that
$|T_1|,|T_2| \geq 2r$, which might be possible by extending \cref{l:2cover} from the class of $2r$-minimal trees to this larger class. We were able to do implement this strategy to show that \cref{conj:main2} holds for metric trees with at most three leaves, using by a more detailed case analysis of the possible ratios between the length of the three branches in \cref{l:2cover}.

 \subsection{General radii}

Given a sequence of radii $(r_1,\dots,r_k)$ what is the maximum $D$ such that $(r_1,\dots,r_k)$ is a cover of every metric tree $T$ with $|T| \leq D$.  By considering intervals (i.e. paths) we see that $D \leq 2 \sum_{i=1}^kr_i$. 
A metric analogue of the Burning Number Conjecture suggests that the equality holds for the sequence $(1,\ldots,k)$, but it is unclear which properties of sequence make the conjecture plausible, motivating the following question.

\begin{question}\label{q:general} Which sequences $(r_1,\dots,r_k)$ of positive reals have the property that $(r_1,\dots,r_k)$ a cover of every metric tree with $|T| \leq 2 \sum_{i=1}^kr_i$?
\end{question}

We believe that the following large and natural class of sequences of radii, which includes the sequences $(1,\ldots,k)$, respects this property. We say a sequence $(r_1,\dots,r_k)$ of non-negative reals is \emph{convex} if $$r_1 \leq r_2 -r_1 \leq r_3 - r_2 \leq \ldots \leq  r_k-r_{k-1}.$$

\begin{conjecture}\label{c:convex}
	If $(r_1,\dots,r_k)$ is a convex sequence of non-negative reals and $T$ is a metric tree such that $|T|\leq 2\sum_{i=1}^k r_i$, then $(r_1,\dots,r_k)$ is a cover of $T$.
\end{conjecture}

\cref{c:convex} is  motivated by the fact that we convinced ourselves that a (rather technical) fractional variant of this conjecture holds asymptotically.\footnote{More precisely the measure $\frac{|T|}{r}\bl{U}[0,r]$ in \cref{t:main2} can be replaced by any measure $\mu$ on $[0,r]$ such that $\m(\mu) \geq |T|$ and $\mu$ has non-increasing density with respect to the Lebesgue measure, i.e. there exists a non-increasing function $f:[0,r] \to \bb{R}_+$  such that $\mu([a,b]) =\int^{b}_{a} f(x)dx$ for all $0 \leq a < b \leq r$.}

The methods of \cref{s:covers} can be used to show that \cref{c:convex} holds for $k=2$.

\begin{thm}\label{t:convex}
	If $0\leq r_1 \leq \frac{r_2}{2}$, and $T$ is a metric tree such that $|T|\leq 2(r_1+r_2)$, then $(r_1,r_2)$ is a cover of $T$.
\end{thm}

\begin{proof}
	We first claim that $T$ is $\frac{2|T|}{3}$-minimal. Trivially, $|T|\geq\frac{2|T|}{3}$, so it suffices to prove there exists that there exists a decomposition $\{T',T''\}$ of $T$ such that $|T'|,|T''|\leq\frac{2|T|}{3}$. Let $\{T',T''\}$ be a decomposition chosen to minimize $\max\{|T'|,|T''|\}$. (Such a choice is possible by compactness and continuity of the length function.) Suppose for a contradiction that that this value is greater than $\frac{2|T|}{3}$, without loss of generality $|T'|>\frac{2|T|}{3}$ . Let $v$ be the unique point of intersection of $T',T''$. Let $\{T_1,\ldots,T_k\}$ be the unique decomposition of $T'$ such that each $T_i$ is a non-trivial branch of $T$ with  anchor $v$ and $v$ is a leaf of $T_i$. If $k=1$, we can slightly nudge $v$ into the segment of $T_1$ incident to $v$, hence slightly decreasing $|T'|$ and slightly increasing $|T''|$. Otherwise $k>1$, and so we can assume without loss of generality that $|T_1|\leq \frac{|T'|}{2}$. Then, consider the decomposition $\{T'\setminus (T_1\setminus \{v\}),T''\cup T_1\}$ of $T$. Then $|T'\setminus (T_1\setminus \{v\})|<|T'|$ and $$|T''\cup T_1|=|T''|+|T_1|\leq\left(|T|-|T'|\right)+\frac{|T'|}{2}=|T|-\frac{|T'|}{2}< |T|-\frac{\frac{2|T|}{3}}{2}=\frac{2|T|}{3} < |T'|,$$ yielding the desired contradiction, and finishing the proof of the claim.
	
	Let $a$ be obtained by applying \cref{l:2cover} to $T$ with $l=\frac{2|T|}{3}$. In particular, $a\leq \frac{\frac{2|T|}{3}}{2}-\frac{|T|}{4}=\frac{|T|}{12}$ and then $\frac{|T|}{4}+a\leq \frac{|T|}{3}\leq \frac{2(r_1+r_2)}{3}\leq r_2$. Hence, there exists $x\geq 0$ such that $\frac{|T|}{4}+a+x=r_2$. If $x \geq \frac{|T|}{4}-3a$ then $(r_2)$ is a cover of $T$ by \cref{l:2cover} and hence so is $(r_1,r_2)$.
	Thus we assume $x \leq \frac{|T|}{4}-3a$ and $$\s{\frac{|T|}{4}-3a-x,\frac{|T|}{4}+a+x} = \s{\frac{|T|}{4}-3a-x,r_2}$$ is a cover of $T$ by \cref{l:2cover}. 
	As $|T|\leq 2(r_1+r_2)=2r_1+\frac{|T|}{2}+2a+2x$, it follows that $r_1\geq \frac{|T|}{4}-a-x\geq \frac{|T|}{4}-3a-x$, and so $(r_1,r_2)$ is a cover of $T$, as desired.		
\end{proof}

In fact, the answer to \cref{q:general} for sequences of length two is exactly the set of convex sequences, i.e.
convexity is not only sufficient, but necessary, for  $(r_1,r_2)$ to be the cover of every metric tree $T$ with $|T|\leq 2(r_1+r_2)$. Consider radii $0\leq r_1\leq r_2$ such that $r_1 > \frac{r_2}{2}$ and let $T$ be metric tree with three leaves, where all three segments have length $\frac{2(r_1+r_2)}{3}$. Firstly, $|T|=2(r_1+r_2)$. We claim that $(r_1,r_2)$ is not a cover of $T$. Suppose otherwise, that $T=B_T(r_1,p_1)\cup B_T(r_2,p_2)$ for some $p_1,p_2\in T$. If $v$ is a leaf of $T$, then $v$ must be at distance at most $r_2$ from $p_1$ or $p_2$. However, the segment of which $v$ is a leaf has length $\frac{2(r_1+r_2)}{3}>\frac{r_2+2r_2}{3}=r_2$. Hence, either $p_1$ or $p_2$ must be strictly contained in this segment, i.e. distinct from $v$. Since $T$ has three branches, it is impossible for every branch to strictly contain $p_1$ or $p_2$, yielding the desired contradiction.

\section*{Acknowledgements} 
We thank Pawe\l{} Rz\k{a}\.{z}ewski for bringing \cite{alon_transmitting_1992} to our attention, and Will Perkins for telling us about point processes. 

\bibliographystyle{abbrvurl}
\bibliography{refs}

\end{document}